\documentclass[11pt]{article}
\usepackage{amsthm}
\usepackage{amssymb}
\usepackage{amsmath,enumerate}
\usepackage{comment}
\usepackage{thm-restate}
\usepackage{url} 
\usepackage{hyperref}
\usepackage{graphicx}
\usepackage{subfigure}
\usepackage[noabbrev,capitalise]{cleveref}
\usepackage[affil-it]{authblk}
\usepackage{color}
\usepackage[normalem]{ulem}

\usepackage[margin=1in]{geometry}

\newcommand{\eps}{\varepsilon}

\theoremstyle{plain}
\newtheorem{thm}{Theorem}[section]
\newtheorem{lem}[thm]{Lemma}

\newtheorem{prop}[thm]{Proposition}
\newtheorem{cor}[thm]{Corollary}

{\noindent \emph{Proof.} {}{#1}{}}{\hfill
	$\Diamond$\vspace{1em}}

\theoremstyle{plain} 
\newcommand{\thistheoremname}{}
\newtheorem{genericthm}[section]{\thistheoremname}

\theoremstyle{definition}

\def\less{\setminus}

\def\dfn#1{{\sl #1}}

\def\less{\setminus}
\def\NN{\mathbb N}
\def\rl{r_\ell}

\title{Multicolor Ramsey  Number for Double Stars}
\author{Jake Ruotolo\thanks{Supported by  the Honors Undergraduate Thesis (HUT) Scholarship at the University of Central Florida.  Current address: School of Engineering and Applied Sciences,
Harvard University, Cambridge,  MA 02134, USA. E-mail address: {\tt  jakeruotolo@g.harvard.edu}.}\,\, and  Zi-Xia Song\thanks{Supported by  NSF grant    DMS-2153945. E-mail address: {\tt Zixia.Song@ucf.edu}.}}
 
  \affil{ 
  { \small {Department  of Mathematics, University of Central Florida, Orlando, FL 32816, USA}}  
     }

\date{}

\begin{document}
\maketitle
\begin{abstract}
For a graph $H$ and an integer $k\ge1$, let $r(H;k)$ and $r_\ell(H;k)$ denote the $k$-color Ramsey number and list Ramsey number of $H$, respectively.  Alon, Buci\'c, Kalvari, Kuperwasser and Szab\'o  in 2021 initiated the systematic study of list Ramsey numbers of graphs and hypergraphs,   and conjectured that  $ r(K_{1,n};k)$ and $r_\ell(K_{1,n};k)$ are always equal. Motivated by their work,  we study the $k$-color Ramsey number for  double stars $S(n,m)$, where $n\ge m\ge1$.  To the best of our knowledge, little is known on the exact value of $r(S(n,m);k)$ when $k\ge3$. 
A classic result of Erd\H{o}s and  Graham from 1975 asserts that $r(T;k)>k(n-1)+1$  for every tree $T$ with $n\ge 1$ edges and  $k$ sufficiently large such that    $n$  divides $k-1$.  Using a folklore double counting argument in set system and the edge chromatic number of complete graphs,  we prove that if  $k$ is odd and $n$ is sufficiently large compared with  $m$ and  $k$, then 
\[ r(S(n,m);k)=kn+m+2.\]
This is a step in our effort to determine whether $r(S(n,m);k)$ and $r_\ell(S(n,m);k)$ are always equal, which remains wide open. 
We also prove that $ r(S^m_n;k)=k(n-1)+m+2$ if  $k $ is odd and $n$ is sufficiently large compared with  $m$ and $k$, where  $1\le m\le n$ and   $S^m_n$ is obtained from $K_{1, n}$ by subdividing $m$ edges each exactly once. We end the paper with some  observations towards  the  list Ramsey number for $S(n,m)$ and $S^m_n$.

\end{abstract}

\baselineskip 16pt

\section{Introduction}
 In this paper we consider graphs that are finite, simple and undirected.      We use $K_{1, n}$  and $K_n$ to denote the star on $n+1$ vertices and  complete graph  on $n$ vertices, respectively.   
The \dfn{double star} $S(n,m)$, where $n\ge m\ge1$,     is the graph consisting of the disjoint union of two stars  $K_{1,n}$  and  $K_{1,m}$ together with an edge  joining their centers.  For integers $n\ge2$ and  $n\ge m\ge1$, let $S_n^m$ denote the graph obtained from $K_{1, n}$ by subdividing $m$ edges each exactly once. Note that $S_n^1=S(n-1, 1)$ and   $S_2^1=S(1, 1)=P_4$,    where $P_4$ denotes the path on four vertices.  
For any positive integer $k$, we write  $[k]$ for the set $\{1,2, \ldots, k\}$. A \dfn{$k$-edge-coloring} of a graph $G$ is a function $\tau:E(G)\to [k]$. We say that $\tau$ is \dfn{proper} if $\tau(e)\ne \tau(e')$  for any two adjacent edges $e, e'$   of $G$. We think of the set  $[k]$ as a set of colors, and we may identify a member of $[k]$ as a color, say, color $k$ is blue. We simply use the term edge-coloring if we do not wish to make reference to the number of colors $k$. 
The $k$-color \dfn{Ramsey number} $r(H; k)$ of a graph  $H$   is the smallest integer $n$ such that every  $k$-edge-coloring of   $K_n$  contains a monochromatic copy of $H$.
In analogy with the well-studied list-coloring version of the chromatic number, Alon, Buci\'c, Kalvari, Kuperwasser, and Szab\'o~\cite{ABKKS2019} recently defined a variant of $r(H;k)$ called the \dfn{list Ramsey number}. Let $L:E(K_n)\to \binom{\NN}{k}$ that assigns a set of $k$ colors to each edge of $K_n$. An \dfn{$L$-edge-coloring of $K_n$} is an edge-coloring where each edge $e$ is given a color in $L(e)$. The $k$-color \dfn{list Ramsey number} $\rl(H;k)$ of a graph $H$ is defined as the smallest $n$ such that there is some $L:E(K_n)\to \binom{\NN}{k}$ for which every $L$-edge-coloring of $K_n$ contains a monochromatic copy of $H$. Taking $L$ to be constant across all edges, we see that \[\rl(H;k)\le r(H;k).\] 
   The authors of \cite{ABKKS2019}
 proved that  $\rl(K_3;k)$ grows  exponential in the square root of $k$.   Fox, He, Luo and Xu~\cite{FHLX2021} continued the work of \cite{ABKKS2019} on bounding multicolor list Ramsey numbers in general. The authors of \cite{ABKKS2019} also investigated when the two Ramsey numbers $\rl(H;k)$ and  $r(H;k)$ are equal, and in general, how far apart they can be from each other. They conjectured that  $\rl(K_{1,n};k) =  r(K_{1,n};k)$ for all $k, n$; and they further posed the question whether $\rl(K_n;k) = r(K_n;k)$ for all $k, n$. The results on $\rl(K_{1,n};k)$  from \cite{ABKKS2019}  are given  in \cref{s:cr}. Motivated by their work in \cite{ABKKS2019} on stars, we aim to   investigate in this paper when the two Ramsey numbers of double stars   $\rl(S(n,m);k)$ and  $r(S(n,m);k)$ are equal.  It is worth noting that the exact value of the $k$-color Ramsey number  $r(K_{1,n};k)$ is known for all $k, n$. However, the exact value of the $2$-color Ramsey number $r(S(n,m);2)$ is not completely known yet, and to the best of our knowledge,   little is known towards the $k$-color Ramsey number $r(S(n,m);k)$ when $k\ge3$. We list the known results on stars and double stars here that we will need later on.


\begin{thm}[Burr and Roberts~\cite{BR73}]\label{t:stars}  For all $k\ge2$ and $n\ge 1$, 
\[r(K_{1,n};k)= \begin{cases}
k(n-1)+1 &\text{ if }  n \text{ and } k \text{ are even}\\
k(n-1)+2   &\text{ otherwise}.
\end{cases}
\]

 \end{thm}

\begin{thm}[Irving~\cite{Irv74}]\label{t:ramseyP4}
For all $k\in \mathbb{N}$, let $\varepsilon$ be the remainder of $k$ when divided by $3$. Then   $r(S(1,1);3)=6$ and 
\[r(S(1,1);k)= \begin{cases}
2k+2  &\text{ if } \varepsilon=1\\
2k+1 &\text{ if } \varepsilon=2\\
2k  \text{ or } 2k+1 &\text{ if } \varepsilon=0.
\end{cases}
\]
\end{thm}

\begin{thm}[Grossman, Harary and Klawe~\cite{GHK79}]\label{t:dstar1}
\[r(S(n,m); 2) = \begin{cases}
\max\{2n+1, n+2m+2\} &\text{ if n is odd and $m\leq 2$},\\
\max\{2n+2, n+2m+2\} &\text{ if n is even or $m\geq 3$, and $n\leq \sqrt{2}m$ or $n\geq 3m$}.
\end{cases}
\]
 For all $n\ge 2$, $r(S(n,1); 2)=2n+2-\varepsilon$, where $\varepsilon$ is the remainder of $n$ when divided by $2$.
\end{thm}

Note that Grossman, Harary and Klawe~\cite{GHK79} further conjectured that the restriction   $n\leq \sqrt{2}m$ or $n\geq 3m$ is not necessary.  Recently,  Norin, Sun and Zhao~\cite{NSZ16} disproved the conjecture for a wide range of values of $m$
and $n$;   using Razborov's flag algebra method,  they confirmed the conjecture when $n\le 1.699(m+1)$. 

\begin{thm}[Norin, Sun and Zhao~\cite{NSZ16}]\label{t:dstar2}
\[r(S(n,m); 2) \ge \begin{cases}
\frac56m+\frac53n+o(m) &\text{ if } n\ge m\ge1,\\
\frac{21}{23}m+\frac{189}{115}n+o(m) &\text{ if } n\ge 2m.
\end{cases}
\]
Furthermore, 
$r(S(n,m); 2)=   
\max\{2n+2, n+2m+2\}=n+2m+2$ when $1\le m\le n\le 1.699(m+1)$. 
\end{thm}

 The fact that   $S(n,m)$ contains $K_{1, n+1}$ as a subgraph yields   the following proposition. 

\begin{prop}\label{p:stardouble} For all $k\ge 2$ and $n\ge m\ge1$, 
\[\rl(S(n,m);k) \ge \rl(K_{1, n+1};k) \text{ and } r(S(n,m);k) \ge r(K_{1, n+1};k).\]
\end{prop}

The main purpose of our paper is to investigate the lower and upper bounds for the $k$-color Ramsey number for double stars (see \cref{s:doublestar}). We prove the following main result.      \medskip

\begin{thm}\label{t:main1}
Let    $k\ge1$ and $n\ge m\ge 1$ be integers satisfying $(n+1)\cdot \left\lceil\frac{n+1}{k-1}\right\rceil> m((k-1)n+m)$.

\begin{enumerate}[(a)]

\item If  $k $ is odd, then 
$ r(S(n,m); k) = kn+m+2$.
 
\item  If   $k $ is even, then  $ \max\{kn+1, (k-1)n+2m+2\}\le r(S(n,m); k) \le kn+m+2$. 
\end{enumerate}
  \end{thm}
   
   A classic result of Erd\H{o}s and  Graham~\cite{EG73} from 1975 asserts that $r(T;k)>k(n-1)+1$  for every tree $T$ with $n\ge 1$ edges and  $k$ sufficiently large such that    $n$  divides $k-1$; in particular, this holds when $T=S(n,m)$. It follows that \cref{t:main1} does not hold for all such $k, n$.   However,   \cref{t:main1} can be extended to subdivided stars $S_n^m$.  \cref{t:main2} below is the second main result     in this paper.

\begin{thm}\label{t:main2}
Let   $k\ge 2$, $n\ge2$ and  $n\ge m\ge 1$ be integers  satisfying $  t>m $ and $nt > (t-m)(m-1)t +m((n-1)(k-1) +m) $, where $t=\left\lceil (n-m+1)/(k-1)\right\rceil$. 

\begin{enumerate}[(a)]

\item If  $k $ is odd, then 
$ r(S_n^m; k) = k(n-1)+m+2$.
 
\item  If   $k $ is even, then  $ \max\{k(n-1)+1, (k-1)(n-1)+2m+2\}\le r(S_n^m; k) \le k(n-1)+m+2$. 
\end{enumerate}
 
\end{thm}

    Our proofs of \cref{t:main1} and \cref{t:main2} are short and utilize  a folklore double counting argument in a set system,  the edge chromatic number of complete graphs,  and a result of K\"onig~\cite{Konig31} from 1931 on the cardinality of maximum matchings and minimum vertex covers of bipartite graphs (only for \cref{t:main2}). 
\medskip
  
  This paper is organized as follows. In the next section, we investigate lower and upper bounds for $r(S(n,m); k)$, and prove \cref{t:main1}. 
  In  \cref{s:substar}, we prove \cref{t:main2}. In
\cref{s:cr},  we present   our observations on the list Ramsey number for stars, double stars and subdivided stars.\medskip

We end this section by introducing more notation.   
Throughout the paper,  we use $(G,\tau)$ to  denote   a $k$-edge-colored  complete graph using colors in $[k]$, where $G$ is a complete graph and $\tau: E(G)\rightarrow [k]$ is a $k$-edge-coloring of $G$ that is not necessarily proper. We say $(G,\tau)$ is \dfn{$H$-free} if $G$ does not contain a monochromatic copy of  a graph $H$ under  the $k$-edge-coloring $\tau$.      
For  two disjoint sets $A, B\subseteq V(G)$,   we  say that  $A$ is \dfn{blue-complete}     to $B$    if all the edges between $A$ and $B$  in $(G,\tau)$ are colored  blue.  We say a vertex $x\in V(G)$ is \dfn{blue-adjacent} to a vertex $y\in V(G)$ if the edge $xy$ is colored blue in $(G,\tau)$. Similar definitions hold when blue is replaced by another color.
Given a graph $H$,    sets $S\subseteq V(H)$ and $F\subseteq E(H)$,  we use   $|H|$    to denote  the  number
of vertices    of $H$,      $H\less S$ the subgraph    obtained from $H$ by deleting all vertices in $S$,   $H\less F$ the subgraph    obtained from $H$ by deleting all edges in $F$,  $H[S] $    the  subgraph    obtained from $H$ by deleting all vertices in $V(H)\less S$, and $H[F]$  the subgraph of $H$ with vertex set $V(H)$ and edge set $F$.  We simply write $H\less v$ when $S=\{v\}$, and $H\less uv$ when $F=\{uv\}$.   We use the convention   ``$A:=$'' to mean that $A$ is defined to be the right-hand side of the relation.  For a positive integer $r$, a graph $H$ is an  \dfn{$r$-factor} of a graph $G$ if $H$ is an $r$-regular subgraph of $G$ such that $V(H)=V(G)$. The \dfn{chromatic index} or \dfn{edge chromatic number} of a graph $G$ is denoted by $\chi'(G)$.

\section{Bounds for $r(S(n,m); k)$}\label{s:doublestar}

In this section, we study lower and upper bounds for $r(S(n,m); k)$.

\subsection{Lower   bounds for $r(S(n,m); k)$}

We begin with lower bound constructions  for $r(S(n,m); k)$ using the chromatic index of complete graphs. In particular, our construction given in   \cref{t:lowbd}(a) is quite simple and nice. 

\begin{thm}\label{t:lowbd}
Let    $k\ge 1$ and $n\ge m\ge 1$  be integers. 
\begin{enumerate}[(a)]

\item  If   $k$ is odd, then  $r(S(n,m); k) \ge kn+m+2$. 
 
\item  If   $k $ is even, then  $r(S(n,m); k) \ge \max\{kn+1, (k-1)n+2m+2\}$. 
\end{enumerate}
\end{thm}
\begin{proof}   To prove (a), it suffices to provide a $k$-edge-coloring $\tau:E(G)\rightarrow [k]$ for the complete graph $G:=K_{kn+m+1}$ such that $(G, \tau)$ is $S(n,m)$-free. This is trivial when $k=1$ by  coloring all edges of $G$ by color $1$. We may assume that $k\ge3$.   Let $H:=K_k$ with $V(H):=\{v_1, \ldots, v_k\}$,  and let $c:E(H)\rightarrow [k]$ be a proper $k$-edge-coloring of $H$. This is possible because $\chi'(K_k)=k$ when $k\ge3$ is odd.  For each $i\in[k]$,   let $c(v_i)$ be the unique color in $[k]$ that does not appear on the edges incident with $v_i$ under the coloring $c$. Then    $c(v_i)\ne c(v_j)$ for  $1\le i<j\le k$. We may assume that $c(v_i)=i$ for each $i\in[k]$.      We now obtain a $k$-edge-coloring $\tau: E(G)\rightarrow [k]$ for $G$ as follows:  first partition $V(G)$ into  $A, V_1, \ldots, V_k$   such that $|A|=m+1$ and $|V_i|=n$ for all $i\in[k]$; then color all  edges of  $G[V_i]$ and all edges between $V_i$ and $A$  by  color $i$ for each $i\in[k]$,     all  edges  between $V_i$ and $V_j$ by  color $c(v_iv_j)$ for  $1\le i<j\le k$, and all  edges of $G[A]$ by color $1$. It is straightforward  to check that $(G, \tau)$ is $S(n,m)$-free, and so $r(S(n,m); k) \ge kn+m+2$, as desired.  This proves (a).
\medskip

  To prove (b), we first observe that $r(S(n,m); k) \ge r(K_{1, n+1}; k)\ge kn+1$ by \cref{t:stars}. We next  show that  $r(S(n,m); k) \ge   (k-1)n+2m+2$. Let $G:=K_{(k-1)n+2m+1}$. We now obtain a $k$-edge-coloring $\tau: E(G)\rightarrow [k]$ for $G$ as follows: first partition $V(G)$ into      $A, B, V_1, \ldots, V_{k-1}$  such that $|A|=m+1$, $|B|=m$, and $|V_i|=n$ for all $i\in[k-1]$. Let $G^*:=G\less B$. Note that $k-1$ is odd and $G^*=K_{(k-1)n+m+1}$. Let  $\tau^*:E(G^*)\rightarrow [k-1]$ be the $(k-1)$-edge-coloring of $G^*$ as constructed in the proof of (a). 
Let $\tau$ be obtained from $\tau^*$ by coloring all edges between $B$ and $V(G)\less B$ by color $k$, and all  edges of $G[B]$ by color $1$.  It is simple to check that $(G,\tau)$ is $S(n,m)$-free, and so  $r(S(n,m); k) \ge   (k-1)n+2m+2$, as desired.   
\end{proof}

 When $k$ is even and  sufficiently large (as a function of  $n+m+1$), we can improve the   bound further in \cref{t:lowbd}(b). We need  the following results of Petersen~\cite{Petersen} on the existence of $2$-factors of regular graphs, and   of Zhang and Zhu~\cite{ZZ92} on $1$-factors of regular graphs. 
 
  \begin{thm}[Petersen~\cite{Petersen}]\label{t:2-factor}
 Every regular graph of positive even degree has a $2$-factor.
\end{thm}

\begin{thm}[Zhang and Zhu~\cite{ZZ92}]\label{t:factor}
 Every $r$-regular graph of order $2n$ contains at least $\lfloor r/2\rfloor$ edge-disjoint $1$-factors if $r\ge n$.
\end{thm}

 \begin{lem}\label{l:lowbd}
Let    $k\ge 1$ and $n\ge m\ge 1$    be integers such that      $k-1$ is divisible  by $n+m+1$.    If $n$ is even,  or  $m$ is odd, then 
\begin{align*}
r(S(n,m); k) \ge kn+m+2. 
\end{align*}
\end{lem}

\begin{proof}
By \cref{t:lowbd}(a), we may assume that $k$ is even. Let  $k:=(n+m+1)\ell+1$ for some integer $\ell\ge1$. Then   $k\ge n+m+2$ and $N := kn+m+1  =(n+m+1)(n\ell+1)$. Let $G:=K_{N}$ and let $\{V_1, \ldots, V_{n\ell+1}\}$ be a partition of $V(G)$ such that $|V_i|=n+m+1$ for all $i\in[n\ell+1]$. Let $H$ be obtained from $G$ by deleting all edges in $G[V_i]$ for each $i\in[n\ell+1]$.  Then $H$ is $(k-1)n$-regular on $N$ vertices. 
We next show that   $E(H)$ can be partitioned into $E_1,   \ldots, E_{k-1}$ such that $ H[E_i] $ is an $n$-factor of $H$ for all $i\in [k-1]$. \medskip

Assume first  $n $ is even. By repeatedly applying \cref{t:2-factor} to $H$,  we see that $E(H)$ can be partitioned into $E_1,   \ldots, E_{k-1}$ such that $ H[E_i] $ is an $n$-factor of $H$ for all $i\in [k-1]$. Assume next $n\ge3$ is odd. Then   $m$ is odd by assumption. Note that    $N=kn+m+1$ is even because $k$ is even;    $(k-1)n\ge N/2$ because $k\ge n+m+2$; in addition,  $(k-1)n/2\ge k-1$ because $n\ge3$. By \cref{t:factor},      $H $ contains at least $k-1$  edge-disjoint $1$-factors, say $F_1,   \ldots, F_{k-1}$. Let $H^*:= H\less \bigcup_{i=1}^{k-1}E(F_i)$. Note that $n-1$ is even and $H^*$ is $(k-1)(n-1)$-regular.  By repeatedly applying \cref{t:2-factor} to $H^*$,   we see that $E(H^*)$ can be partitioned into $E'_1,  \ldots, E'_{k-1}$ such that $H[E'_i] $ is an $(n-1)$-factor of $H$ for each $i\in [k-1]$. Let $E_i:=E'_i\cup F_i$ for each $i\in[k-1]$. Then 
$H[E_i] $ is an $n$-factor of $H$ for each $i\in [k-1]$.  \medskip

Now coloring  all  edges of $G[V_j]$ by color $k$ for each $j\in[n\ell+1]$, and all edges of   $E_i$ by color $i$ for each $i\in[k-1]$, we obtain a $k$-edge-coloring $\tau$ of $ G$ such that $(G,\tau)$ is $S(n,m)$-free.  Therefore, $r(S(n,m); k) \ge kn+m+2$. 
\end{proof}

The proof of \cref{l:lowbd1} is similar to the proof of \cref{l:lowbd}. We provide a proof here for completeness.

\begin{lem}\label{l:lowbd1}
Let  $k\ge 2$ and $n\ge m\ge 1$ be integers such that  $n$ is even, $m$ is odd, and  $k-1$  is divisible by $\frac{n+m+1}2$. Then 
\begin{align*}
r(S(n,m); k) \ge kn+m+2. 
\end{align*}
\end{lem}
\begin{proof} Let $k:=\frac{n+m+1}2\ell+1$ for some integer $\ell\ge1$.  Then $N := kn+m+1=  (n+m+1)(\frac{n\ell}2+1)$. Let      $p:=\frac{n\ell}2+1$. Then $p$ is a positive integer because $n$ is even. Let $G:=K_{N}$ and let $\{V_1, \ldots, V_p\}$ be a partition of $V(G)$ such that $|V_i|=n+m+1$ for all $i\in[p]$. Let $H$ be obtained from $G$ by deleting all edges in $G[V_i]$ for each $i\in[p]$.  Then $H$ is $(k-1)n$-regular on $N $ vertices. Note that $(k-1)n$ is even. By repeatedly applying \cref{t:2-factor} to $H$,  we see that $E(H)$ can be partitioned into $E_1,   \ldots, E_{k-1}$ such that $ H[E_i] $ is an $n$-factor of $H$ for each $i\in [k-1]$. We now obtain a $k$-edge-coloring $\tau$ of $G$ by  coloring  all  edges of $G[V_j]$ by color $k$ for each $j\in[p]$, and all edges of   $E_i$ by color $i$ for each $i\in[k-1]$. Then $(G,\tau)$ is $S(n,m)$-free, and so  $r(S(n,m); k) \ge kn+m+2$. 
\end{proof}

\subsection{Upper bounds for $r(S(n,m); k)$}
We next show that the lower bound in \cref{t:lowbd}(a) is sharp for all  $k\ge 3$ odd and $n $ sufficiently large.
We need Lemma~\ref{p(x)}. Its proof follows from a simple double counting argument and can be found in \cite[Proposition 1.7]{Juk11}.

\begin{lem}\label{p(x)} Let $\mathcal{F}$ be a family of subsets of some set $X$. For each $x\in X$,  we  define    $p(x)$  to be   the number of members of $\mathcal{F}$  containing $x$.  Then

\[\sum_{x\in X} p(x)=\sum_{F\in\mathcal{F}}|F|.\]
 
\end{lem}

\begin{thm}\label{t:ubds}
Let   $k\ge 2$ and $n\ge m\ge 1$ be integers. If $(n+1)\cdot \left\lceil\frac{n+1}{k-1}\right\rceil> m((k-1)n+m)$, then  
\begin{align*}
r(S(n,m); k) \leq kn+m+2.
\end{align*}
\end{thm}

\begin{proof}
Let $(G,\tau)$ be a complete, $k$-edge-colored $K_{kn+m+2}$ using colors in $\left[k\right]$. Then $G$ contains a monochromatic copy of $H:=K_{1,n+1}$, say in color $k$.  We may assume that the color $k$ is blue.  Let $A:=\{a_1, \ldots, a_{n+1}\}$ be the set of $n+1$ leaves of $H$, that is, the set of vertices of degree one in $H$,  and let $B := V(G)\setminus V(H)$.  Then $|A|=n+1$ and $|B|=(kn+m+2)-(n+2)=(k-1)n+m$. We may assume that  each vertex in  $ A$ is  blue-adjacent to at most  $m-1$    vertices  in $B$, otherwise we are done.  For each $a_i\in A$, let $E_i:=\{a_ib\mid b\in B \text{ and } \tau(a_ib)\ne k \}$.  Then $|E_i|\ge |B|-(m-1)= (k-1)n+1$, and all the edges in $E_i$ are colored using colors in $\left[k-1\right]$ under $\tau$. By the pigeonhole principle, each $a_i\in A$ is the center of a monochromatic copy of $H_i:=K_{1,n+1}$,   in  some color in $[k-1]$, with leaves in $B$.  Since $|A|=n+1$, we see that    at least $t:=\lceil(n+1)/(k-1)\rceil$  of $H_1, H_2, \ldots, H_{n+1}$, say  $H_1, H_2, \ldots,  H_t$,    are colored the same by some   color in   $ [k-1]$. We may further assume that  $H_1, H_2, \ldots, H_t$ are in color red. Let $L_i$ be the set of  leaves of $H_i$ for each $i\in[t]$.  Let  $\mathcal{F} := \{L_1,   \ldots, L_{t}\}$. For $b\in B$,  let $p(b)$ be defined as in Lemma~\ref{p(x)}.  Let $b^*\in B$ such that  $p(b^*)$ is maximum. By Lemma~\ref{p(x)} and the choice of $n,m,k$, we have 
\[((k-1)n+m)\cdot p(b^*)= |B|\cdot p(b^*)\ge \sum_{b\in B} p(b)=\sum_{L\in\mathcal{F}}|L|= (n+1)\cdot \left\lceil\frac{n+1}{k-1}\right\rceil>m((k-1)n+m). \]
It  follows that $p(b^*)\ge m+1$. We may further assume that $b^*\in L_1\cap \cdots\cap L_{m+1}$.   Then $(G,\tau)$ contains  a red copy of  $S(n,m)$   with its edge set $E(H_{m+1})\cup\{b^*a_i\mid i\in [m]\}$, as desired.  
\end{proof}

 Note that $ r(S(n,m); 1) = n+m+2$. Combining this with \cref{t:lowbd} and \cref{t:ubds}
leads to  our main result,  \cref{t:main1}.
For  all $k\ge3$ and $m=1$, we can improve    the bound for $n$ in \cref{t:ubds}.  \cref{l:ubm=1} follows from the proof of  \cref{t:ubds}. We provide a proof here for completeness.

\begin{lem}\label{l:ubm=1}
Let $k\geq 3$ and $n\geq (k-1)(k-2)$ be integers. Then $r(S(n,1); k) \leq kn+3$.
\end{lem}

\begin{proof}
Let $(G,\tau)$ be a complete, $k$-edge-colored $K_{kn+3}$ using colors in $\left[k\right]$.    Let $v\in V(G)$. Then $v$ is the center of  a monochromatic copy of   $H:=K_{1,n+1}$, say  in  color $k$.  Let $A:=\{v_1, v_2, \ldots, v_{n+1}\}$ be the leaves of $H$.  
Let $B := V(G)\setminus \{v, v_1, \ldots v_{n+1}\}$.  We may   assume that no edge between $A$ and $B$ is  colored by color $k$, otherwise we are done.  Thus all the edges between $A$ and $B$ are colored using the colors in  $\left[k-1\right]$. Note  that $|A|=n+1\ge(k-1)(k-2)+1$ and  $|B|=(k-1)n +1$.  It follows that each $v_i$ is the center of a monochromatic copy of $H_i:=K_{1,n+1}$,   in  some color in $[k-1]$, with leaves in $B$;  at least $\lceil |A|/(k-1)\rceil\ge k-1$ of such stars $H_1, \ldots, H_{n+1}$ are colored by the same color in $[k-1]$, say in color red; and at least   two of such $k-1$ red stars $K_{1,n+1}$ share   one leaf in common, since $k\ge3$. Therefore, $(G,\tau)$ contains a red copy of $S(n,1)$, as desired.
\end{proof}

\begin{cor}\label{c:m=1}
Let $k\geq 3$   and $n\geq (k-1)(k-2)$ be integers. 
\begin{enumerate}[(a)]
\item If $k$ is odd, then  $r(S(n,1); k) = kn+3$. In particular, $r(S(n,1); 3)=3n+3$ for all $n\ge1$. 
\item If both $k$ and $n$ are even, then $kn+2  \leq r(S(n,1); k) \leq kn +3$.
\item If $k$ is  even and $n$ is odd,   then $kn+1 \leq r(S(n,1); k) \leq kn +3$.
\end{enumerate}
\end{cor}
 
 \begin{proof}
If $k\ge3$ is odd, then  $r(S(n,1); k) =  kn+3$ by \cref{t:lowbd} and \cref{l:ubm=1}. By \cref{t:ramseyP4}, $r(S(1,1); 3) = 6$, and so  $r(S(n,1); 3)=3n+3$ for all $n\ge1$.  Next, if $k$  is even,  by \cref{t:stars} and \cref{l:ubm=1}, we see that  $kn+2 = r(K_{1,n+1}; k) \leq r(S(n,1); k) \leq kn+3$  if $n$ is even, and $kn+1= r(K_{1,n+1}; k) \leq r(S(n,1); k) \leq kn +3$  if $n$ is odd. 
\end{proof}

 \section{Bounds for $r(S_n^m; k)$}\label{s:substar}
 In this section we prove \cref{t:main2}.  
Recall that  $S_n^m$ denotes the graph obtained from $K_{1, n}$ by subdividing $m$ edges each exactly once, where   $n\ge2$ and  $n\ge m\ge1$. Note that $S_n^1=S(n-1, 1)$,  $S_2^1=S(1, 1)=P_4$,  and $r(S^m_n; k)\ge r(K_{1,n}; k)$ for all $k\ge2$. 
 \cref{t:substar-low} below follows directly from the  proof of  \cref{t:lowbd} by letting  $|A|=m+1$ and  $|V_1|=\cdots=|V_k|=n-1$ when $k$ is odd, and letting   $|A|=m+1$, $|B|=m$ and $|V_1|=\cdots=|V_{k-1}|=n-1$  when $k$ is even. We omit the proof here.  
 
\begin{thm}\label{t:substar-low}
Let $k\ge2$,  $n\geq 2$ and $n\ge m\geq 1$ be integers. \begin{enumerate}[(a)]

\item If  $k $ is odd, then 
$ r(S_n^m; k) \ge k(n-1)+m+2$.
 
\item  If   $k $ is even, then  $ r(S_n^m; k) \ge \max\{k(n-1)+1, (k-1)(n-1)+2m+2\}$. 
\end{enumerate}

\end{thm}

 Our proof of \cref{t:substar-up}   follows the main idea in the proof of \cref{t:ubds} but more involved. We need both    \cref{p(x)} and   a  result of K\"onig from 1931.  Note that our second main result \cref{t:main2} follows from \cref{t:substar-low} and \cref{t:substar-up}.
 
\begin{thm}[K\"onig~\cite{Konig31}]\label{t:mmmc}
Let $G$ be a bipartite graph. Then the  maximum cardinality of a matching in $G$ is equal to the minimum cardinality of a vertex cover in $G$. 
\end{thm}

\begin{thm}\label{t:substar-up}
Let   $k\ge 2$, $n\ge2$ and  $n\ge m\ge 1$ be integers and let $t=\left\lceil (n-m+1)/(k-1)\right\rceil$. If $  t>m $ and $nt > (t-m)(m-1)t +m((k-1)(n-1) +m)  $, then 
\[r(S_n^m; k) \leq k(n-1)+m+2.\]
\end{thm}

\begin{proof}
Let $(G,\tau)$ be a complete, $k$-edge-colored $K_{k(n-1)+m+2}$ using colors in $\left[k\right]$. Then $G$ contains a monochromatic copy of  $H:=K_{1,n}$, say in color $k$.  We may  assume that the color $k$ is blue.  Let $A:=\{a_1, \ldots, a_{n}\}$ be the set of $n$ leaves of $H$ and let $B := V(G)\setminus V(H)$.  Then \[|B|=(k(n-1)+m+2)-(n+1)=(k-1)(n-1)+m.\]
 Let $G^*$ be the bipartite graph with   $V(G^*)=A\cup B$ and   $E(G^*)$ consisting of all blue edges between $A$ and $B$ in $G$ under the coloring $\tau$.    Then $G^*$ has no  matching of size $m$, otherwise we are done.    Let $C\subseteq V(G^*)$ be a minimum vertex cover of $G^*$. By \cref{t:mmmc}, $|C|\le m-1$. Let $A':=A\less C$ and $B':=B\less C$.
Then $|A'|\ge  n-(m-1)$ and $|B'|\ge (k-1)(n-1)+1$. Now all the edges  between $A'$ and $B'$  are colored using colors in $[k-1]$ under $\tau$. We may assume that $ a_1, \ldots, a_{n-m+1}\in A'$. By the pigeonhole principle, each $a_i\in A'$ is the center of a monochromatic copy of  $H_i:=K_{1,n}$,   in  some color in $[k-1]$, with leaves in $B'$.  Since $|A'|\ge n-m+1$, we see that   there are at least $t:=\left\lceil (n-m+1)/(k-1)\right\rceil>m$  of $H_1,   \ldots, H_{n-m+1}$, say  $H_1,   \ldots, H_t$,    are colored the same by some   color in   $ [k-1]$. We may further assume that  $H_1,   \ldots, H_t$ are in color red. Let $L_i$ be the set of  leaves of $H_i$ for each $i\in[t]$.  Let  $\mathcal{F} := \{L_1,   \ldots, L_{t}\}$. For $b\in B'$,  let $p(b)$ be   the number of members of $\mathcal{F}$ containing $b$.   For each $i \in [t]$, define $L^*_i = \{x\in L_i\mid p(x)\geq m+1\}$. We next show that $|L^*_j|\ge m$  for some $j\in[t]$.  
Suppose  $|L^*_i|\le m-1$ for each $i\in [t]$.  Let $B^*:=\bigcup_{i=1}^{t} L^*_i$.  
Then $|B^*|\le (m-1)t$. Note that $p(b)\le t$ for each $b\in B^*$,   $p(b)\le m$ for each $b\in B'\less B^*$, and $|B'|\le|B|= (k-1)(n-1) +m$. It follows that  \medskip
\begin{align*}
\sum_{b\in B'}p(b)&= \sum_{b\in B^*}p(b) + \sum_{b\in B'\setminus B^*}p(b)\\
&\le t|B^*|+m(|B'|-|B^*|)\\
&=(t-m)|B^*|+m|B'|\\
&\le (t-m)(m-1)t+m((k-1)(n-1)+m).
\end{align*}
However, by   Lemma~\ref{p(x)}, we have 
\[\sum_{b\in B'}p(b)=\sum_{F\in\mathcal{F}}|F|=\sum_{i=1}^t |L_i|=nt,\]
contrary to the assumption that   $nt > (t-m)(m-1)t +m((k-1)(n-1) +m)$. Thus   $|L^*_j|\ge m$ for some $j\in [t]$, say $j=1$. Let $b_1, \ldots,  b_m \in L^*_1$. Then  $p(b_i)\ge m+1$ for each $i\in [m]$.  By assumption, we have $t>m$.  We may further assume that $b_i\in L_{i+1}$ for each $i\in [m]$.   Then  $(G, \tau)$ contains a red copy of $S_n^m$ with its edge set $E(H_1)\cup\{b_1a_2,\ldots, b_ma_{m+1}\}$, as desired.
\end{proof}

\section{Some results on list Ramsey number}\label{s:cr}

As mentioned in the Introduction, our motivation of this paper is to determine whether  $\rl(S(n, m);k)$ and $r(S(n, m);k)$ are always equal. This seems far from trivial. We end this paper with our  observations towards $\rl(K_{1,n};p)$ and $\rl(S(1,1);p)$ for every odd prime $p$,   and $\rl(S(n, m);2)$ and $\rl(S_n^m;2)$. The authors of~\cite{ABKKS2019} proved the following important result on $\rl(K_{1,n};k)$. 
 
\begin{thm}[Alon, Buci\'c, Kalvari, Kuperwasser, and Szab\'o~\cite{ABKKS2019}]\label{t:liststars}
For any $k$ and $n\in \mathbb{N}$, except
    possibly finitely many integers $n$ for each   $k$, we have
    $\rl(K_{1,n}; k)= r(K_{1,n};k)$. More precisely,
\begin{itemize} 
\item[(a)] For every $k,n\in \mathbb{N}$, we have  $k(n-1) + 1\leq  \rl(K_{1,n};k)$.
In particular, if both $n$ and $k$ are even, then \[\rl(K_{1,n}; k) = k(n-1) +1= r(K_{1,n}; k).\]
\item[(b)]For every $k\in \mathbb{N}$ there exists $w(k)\in \mathbb{N}$
  such that the following holds. For every $k$ and $n \geq w(k)$
  that are not both even, we have $$\rl(K_{1,n};k) = k(n-1)+2 = r(K_{1,n};k).$$
\end{itemize}
\end{thm}

Following the proof of \cite[Lemma 11]{ABKKS2019}, one  can prove \cref{t:oddprimestar} below applying  \cref{t:oddprime}.  We recall  the  proof here for completeness.

\begin{thm}[Galvin~\cite{Galvin}]\label{t:bip} If $G$ is a bipartite graph, then  
$\chi'_{\ell}(G) \le \Delta(G)$.
\end{thm}
    \begin{thm}[Schauz~\cite{Schauz14}]\label{t:oddprime}
$\chi'_{\ell}(K_{p+1}) = p=\chi'(K_{p+1})$ for every odd prime $p$.
\end{thm}

 \begin{thm}\label{t:oddprimestar}
For all $n\geq 2$ and every odd prime $p$, 
\[r_{\ell}(K_{1,n}; p) = r(K_{1,n};  p) = p(n-1)+2.\]
\end{thm}
  \begin{proof} By Theorem~\ref{t:stars}, we see that  $r_{\ell}(K_{1,n}; p) \le  r(K_{1,n};  p)\le p(n-1)+2$.  It suffices to show that $r_{\ell}(K_{1,n}; p) > p(n-1)+1$.  Let  $G = K_{p(n-1)+1}$.      Partition $ V(G)$ into $A, V_1, V_2, \ldots V_{n-1}$  such that $|A|=1$ and  $|V_1|=\cdots=|V_{n-1}|=p$.   For $i, j \in \left[n-1\right]$ with $i< j$, let $G_{i,j}$ be the complete bipartite subgraph    with partite sets $V_i$ and $ V_j $.  By Theorem~\ref{t:bip}  and Theorem~\ref{t:oddprime},    $\chi'_{\ell}(G_{i,j})=p$ and $\chi'_{\ell}(G[A\cup V_i])=p$. Note that   the   vertex in $A$  belongs to   $n-1$  subgraphs $G[A\cup V_1], \ldots, G[A\cup V_{n-1}]$,  and  each vertex in $V_1\cup\cdots\cup V_{n-1}$   belongs to exactly one of subgraphs  $G[A\cup V_1], \ldots, G[A\cup V_{n-1}]$ and    $n-2$  of the subgraphs $G_{i,j}$'s. By  \cite[Lemma 10]{ABKKS2019}, $r_{\ell}(K_{1,n}; p) > p(n-1)+1$, as desired. 
  \end{proof}
It is worth noting that  \cref{t:liststars}(a) fails to give a full characterization of the tightness of
the lower bound   but for $k=2$, the authors of~\cite{ABKKS2019}   gave such a
characterization and proved  that the two Ramsey numbers
are always equal.
\begin{thm}[Alon, Buci\'c, Kalvari, Kuperwasser, and Szab\'o~\cite{ABKKS2019}]\label{t:stars-2-col}
For every $n\in \mathbb{N}$ we have
$$\rl(K_{1,n};2)=r(K_{1,n};2)=
\begin{cases}
2n  &\text{ if } n\text{ is odd}\\
2n-1 &\text{ if } n \text{ is even.}

\end{cases}
$$
\end{thm}

By \cref{t:dstar1} and \cref{t:stars-2-col}, together with \cref{p:stardouble}, we see that for all $n\ge2$, \[\rl(S(n,1);2)=r(S(n,1);2)=r(K_{1, n+1};2) =\rl(K_{1, n+1};2)=\begin{cases}
2n+1  &\text{ if } n\text{ is odd}\\
2n+2 &\text{ if } n \text{ is even.}
\end{cases}
\tag{$*$}\] 

Moreover, for all odd $n \ge 5$, we have
 \[\rl(S(n,2);2)=r(S(n,2);2)=r(K_{1, n+1};2) =\rl(K_{1, n+1};2)=2n+1.\] 
For all $n \ge 3m$ such that $n$ is even or $m \ge 3$, we have
\[2n+2 =r(S(n,m);2)\ge \rl(S(n,m);2)\ge\rl(K_{1, n+1};2)= r(K_{1, n+1};2)=\begin{cases}
2n+1  &\text{ if } n\text{ is odd}\\
2n+2 &\text{ if } n \text{ is even.}
\end{cases}
\] 
Liu~\cite{HenryLiu} proved that $\rl(G;2)=r(G;2)$ for every graph  $G\in\{P_4, P_5, C_4\}$.  By \cref{p:stardouble}, we have $\rl(S(n, m);k)\ge \rl(K_{1, {n+1}};k)\ge kn+1$ due to   \cref{t:liststars}. Also,  $\rl(S(n, m);p)\ge \rl(K_{1, {n+1}};p)\ge pn+2$  for every odd prime $p$ due to    \cref{t:oddprimestar}. This, together with \cref{c:m=1}(a), implies that for all $n\ge(p-1)(p-2)$, \[pn+2=  \rl(K_{1, {n+1}};p)\le \rl(S(n, 1);p)\le r(S(n, 1);p) =pn+3.\tag{$\dagger$}\]  We are unable to   determine which side $\rl(S(n, 1);p)$ lies on in ($\dagger$). Note  that $r_{\ell}(S(1,1); k)\ge k+1$ by  \cref{p:stardouble} and \cref{t:liststars}(a).  We next prove a slightly improved lower bound for $r_{\ell}(S(1,1); p)$ for every odd prime $p$. Recall that  $P_4=S(1,1)$.

\begin{lem}\label{l:k-colorP4}
$r_{\ell}(P_4; p)  \ge p+3$  for every  odd prime $p$. 
\end{lem}

\begin{proof}
   Let $G:=K_{p+2}$.  Let $L: E(G)\to \binom{\mathbb{N}}{p}$ be an assignment of lists to the edges of $G$.  If $L$ is constant, then we are done by \cref{t:ramseyP4}. We may assume that there exists a vertex, say $u$, in $G$ such that $\left|\bigcup_{v\in N(u)} L(uv)\right| \geq p+1$. Now color the edges incident with $u$ differently. Since $\chi'_{\ell}( K_{p+1}) = p$ by \cref{t:oddprime},  we can color the edges of  $G\less u$   from $L$ such that it has no monochromatic $K_{1,2}$.  It follows that $G$ has no   monochromatic copy of $P_4$, as desired. 
\end{proof}

 \cref{l:k-colorP4}, combined with \cref{t:ramseyP4}, gives the exact value of $r_{\ell}(P_4; 3)$. 
 
\begin{cor}
$r_{\ell}(P_4; 3) =r(P_4; 3)= 6$.   
\end{cor}

We end this section with an observation towards  $ \rl(S_n^m;2)$ when $m\in\{2,3\}$. Note that  for all $n\ge2$, we have $S_n^1=S(n-1,1)$;    $\rl(S_n^1;2)=\rl(S(n-1, 1);2)=r(K_{1,n};2)$ by ($*$).

\begin{thm}\label{t:m=23}Let $m\in\{2,3\}$ and $n\ge3m-1+\varepsilon$ be integers, where $\varepsilon$ is the remainder of $n-1$ when divided by $2$. Then \[\rl(S_n^m;2)=r(S_n^m;2)=r(K_{1,n};2)=\rl(K_{1,n};2)=
\begin{cases}
2n  &\text{ if } n\text{ is odd}\\
2n-1 &\text{ if } n \text{ is even.}
\end{cases}\] 
\end{thm}
\begin{proof}  Let $n, m$ and $\varepsilon$ be given as in the statement.  Note that  $r(S_n^m;2)\ge \rl(S_n^m;2)\ge \rl(K_{1,n};2)$. By \cref{t:stars-2-col}, it suffices to show that $r(S_n^m;2)\le  r(K_{1,n};2)$.  By \cref{t:stars}, we have $r(K_{1,n};2)=2n-\varepsilon$.   Let $(G,\tau)$ be a complete, $2$-edge-colored $K_{N}$ using colors  red and blue, where  $N=2n-\varepsilon$. Since $S_n^1=S(n-1, 1)$, by \cref{t:dstar1}, we have $r(S_n^1;2) = 2n-\eps = N$ for $n \ge3m-1+\eps\ge 5+\eps$. Thus, $(G,\tau)$ contains a monochromatic copy of $S_n^1$. Suppose $(G,\tau)$  is $S_n^m$-free. We choose $(G,\tau)$ such that $(G,\tau)$ is $S_n^\ell$-free and $\ell$ is minimum. Then $2\le\ell\le m$ and     $(G,\tau)$  must contain a monochromatic copy of $H:=S_n^{\ell-1}$, say in color red.      Let $V(H):=\{x, y_1, \ldots, y_n, z_1, \ldots, z_{\ell-1}\}$ such that $E(H)=\{xy_1, \ldots, xy_n, y_1z_1, \ldots, y_{\ell-1}z_{\ell-1}\}$. Let $A:=\{y_\ell, \ldots, y_n\}$,   $B := V(G)\setminus V(H)$, and $C:=V(H)\less A$.  Then all edges  between $A$ and $B$ are colored blue. Note that $|B|=N-|V(H)|=(2n-\varepsilon)-(n+\ell)=n-\ell-\varepsilon\ge 2m-1$. Let $ B:=\{b_1, \ldots, b_{n-\ell-\varepsilon}\}$.\medskip

We first consider the case $\ell=2$. Then  $n\ge5$ and   $|B| =n-2-\varepsilon\ge  3$.    Suppose some vertex of $B$, say $b_1 \in B$,  is blue-adjacent to some vertex $u\in C$.   Then $(G,\tau)$ contains a blue $S_n^2$ with edge set $\{b_1u, b_1y_2, \ldots, b_1y_n, y_nb_2, y_{n-1}b_3\}$, a contradiction.  Thus $B$ is red-complete to $ \{x, y_1, z_1\}$.    But then  we obtain a red $S_n^2$ with edge set $\{xb_1, b_1y_1, xb_2, b_2z_1, xy_2, \ldots, xy_{n-1}\}$, a contradiction.\medskip

We next  consider the case $\ell=3$.      Then  $n\ge 8$ and  $|B|=n-3-\varepsilon\ge  5$.  We claim that  every vertex in $B$ is blue-adjacent to exactly one vertex in $C$. Suppose, say $b_1\in B$, is red-complete to $C$ or blue-adjacent to   two distinct vertices, say, $u_1, u_2$,  in $C$. In the former case, $b_1$ is blue-complete to $B\less b_1$. Also,  $\{z_1, z_2\}$ is blue-complete to $A$, else we have a red $S_n^3$.   But then $(G,\tau)$ contains a blue $S_n^3$ with edge set $\{b_1b_2, b_1 b_3, b_1y_3, \ldots, b_1y_n, y_nz_1, y_{n-1}z_2, y_{n-2}b_4\}$, a contradiction. In the latter case,   $(G,\tau)$ contains a blue $S_n^3$ with edge set $\{b_1u_1, b_1u_2, b_1y_3, \ldots, b_1y_n, y_nb_2, y_{n-1}b_3, y_{n-2}b_4\}$, a contradiction.  Thus  every vertex in $B$ is blue-adjacent to exactly one vertex in $C$, as claimed. It follows that all edges in $G[B]$ are colored red, else say $b_1b_2$ is colored blue. Since every vertex in $B$ is blue-adjacent to exactly one vertex in $C$, we may assume that $b_1$ is blue-adjacent to $u\in C$. But then  $(G,\tau)$ contains a blue $S_n^3$ with edge set $\{b_1u, b_1b_2,   b_1y_3, \ldots, b_1y_n, y_nb_3, y_{n-1}b_4, y_{n-2}b_5\}$, a contradiction. This proves that all edges in $G[B]$ are colored red. Then $x$ is blue-complete to $B$, and  $B$ is red-complete to $C\less x$  because every vertex in $B$ is blue-adjacent to exactly one vertex in $C$. Thus    $\{z_1, z_2\}$ is blue-complete to $A$, else we obtain a red $S_n^3$.     
Finally, suppose some vertex of $A$, say $y_3\in  A$,  is blue-complete to $\{y_1, y_2\}$.  Then $(G,\tau)$ contains a blue $S_n^3$ with edge set $\{y_3y_1, y_3y_2,    y_3z_1, y_3z_2, y_3b_1,\ldots, y_3b_{n-4}, b_1y_4, b_2y_5, b_3 y_6\}$, a contradiction.  
Thus  no vertex in $A$  is blue-complete to $\{y_1, y_2\}$. It follows that either 
$y_1$ or $y_2$ is red-adjacent to at least $|A|/2=(n-2)/2\ge 3$ vertices in $A$. We may assume that $y_1$ is red-complete to $\{y_3, y_4, y_5\}$.  But then  $(G,\tau)$ contains a red $S_n^3$ with edge set   \[\{y_1x,   y_1y_3, y_1y_4, y_1y_5, y_1b_1, \ldots, y_1b_{n-4},  xy_2, b_1z_1, b_2z_2\},\] a contradiction.  \end{proof}

It seems that   our proof method of \cref{t:m=23}  can be extended to  show that $\rl(S_n^m;2)=r(S_n^m;2)$  when $n$ is sufficiently larger compared to $m$ for all $m\ge4$. However, \cref{t:m=23}  does not hold when $n\le 2m-1$.

\begin{lem}
 For all $n\ge2$ and $n\ge m\ge  1$, we have  $r(S^m_n; 2) \ge n+2m$. In particular, 
 \[r(S^m_n; 2)>2n\ge r(K_{1, n};2)\] for all  $n\le 2m-1 $. 
\end{lem}

\begin{proof} Let $G: =K_{n+2m-1}$. We partition the vertex set of $G$ into     $A$ and $ B$   such that $|A|=n+m$ and $|B|=m-1$. Let $\tau$ be a $2$-edge-coloring of $G$ by coloring all  edges in $G[A]$ and $G[B]$ red,  and all  edges   between $A$ and $B$ blue.  It is simple  to check that $(G,\tau)$ is  $S_n^m$-free. Therefore,  $r(S^m_n; 2) \ge n+2m$, as desired. By \cref{t:stars}, we have $r(S^m_n; 2)>2n\ge r(K_{1, n};2)$ for all $n\le 2m-1$. 
\end{proof}

\section*{Acknowledgements}
 We are indebted to  anonymous referees for their valuable comments which greatly improve the presentation of this paper. The second author would like to thank Henry Liu for sending her the reference  \cite{HenryLiu}.

\end{document}